\documentclass[reqno]{amsart}
\usepackage[scale=0.75, centering, headheight=14pt]{geometry}
\usepackage[latin1]{inputenc}
\usepackage[T1]{fontenc}
\usepackage{lmodern}
\usepackage[english]{babel}

\usepackage{stmaryrd}
\usepackage{tikz}
\usepackage{bbm}
\usepackage{amsmath,amssymb,amsfonts,amsthm}
\usepackage{mathtools,accents}
\usepackage{mathrsfs}
\usepackage{xfrac}
\usepackage{array} 
\usepackage{aliascnt}
\usepackage{booktabs} 
\usepackage{array} 

\usepackage{verbatim} 
\usepackage{subfig} 
\usepackage{mathrsfs}
\usepackage{mathrsfs, dsfont}
\usepackage{amssymb}
\usepackage{amsthm}
\usepackage{amsmath,amsfonts,amssymb,esint}
\usepackage{graphics,color}
\usepackage{enumerate}
\usepackage{mathtools,centernot}

\usepackage{microtype}
\usepackage{paralist} 
\usepackage{cases}
\usepackage[initials]{amsrefs}
\allowdisplaybreaks

\usepackage{braket}
\usepackage{bm}

\usepackage[citecolor=blue,colorlinks]{hyperref}
\addto\extrasenglish{}

\usepackage{enumerate}
\usepackage{xcolor}

\usepackage{aliascnt}

\makeatletter
\def\newaliasedtheorem#1[#2]#3{
  \newaliascnt{#1@alt}{#2}
  \newtheorem{#1}[#1@alt]{#3}
  \expandafter\newcommand\csname #1@altname\endcsname{#3}
}
\makeatother

\theoremstyle{plain}
\newtheorem{thmx}{Theorem}

\newtheorem*{conx}{Open Question}
\newtheorem{theorem}{Theorem}[section]
\newaliasedtheorem{lemma}[theorem]{Lemma}
\newaliasedtheorem{prop}[theorem]{Proposition}
\newaliasedtheorem{claim}[theorem]{Claim}
\newaliasedtheorem{corollary}[theorem]{Corollary}

\theoremstyle{remark}

\theoremstyle{definition}
\newaliasedtheorem{Def}[theorem]{Definition}
\newaliasedtheorem{example}[theorem]{Example}

\theoremstyle{remark}
\newaliasedtheorem{remark}[theorem]{Remark}
\newaliasedtheorem{ex}[theorem]{Example}

\numberwithin{equation}{section}

\def\eps{\varepsilon}
\DeclareMathOperator{\tr}{tr}
\def\R{\mathbb R}
\def\N{{\mathbb N}}

\DeclareMathOperator{\dv}{div}
\DeclareMathOperator{\Lip}{Lip}

\DeclareMathOperator{\curl}{curl}

\DeclareMathOperator{\rank}{rank}
\DeclareMathOperator{\diam}{diam}

\DeclareMathOperator{\Sym}{Sym}
\DeclareMathOperator{\co}{co}
\DeclareMathOperator{\cof}{cof}
\DeclareMathOperator{\loc}{loc}
\DeclareMathOperator{\id}{id}

\DeclareMathOperator{\spt}{spt}

\DeclareMathOperator{\dist}{d}

\title{Rigidity of shear flows of the Euler equations in the plane}
\author[R. Tione]{Riccardo Tione}

\begin{document}
	
\begin{abstract}
In this paper we show that steady states $u$ of the pressureless Euler equation which belong to $L^3_{\loc}(\mathbb{R}^2,\mathbb{R}^2)$ are shear flows.\ This is achieved by combining results of degenerate Monge-Amp\`ere-type equations with the theory of two dimensional transport equations.\ We also show that the problem of rigidity and flexibility for the associated differential inclusion is rigid for sequences equibounded in $L^{4+}$ and flexible for sequences equibounded in $L^{4-}$, thus displaying a gap in the rigidity exponent between the exact and the approximate problem.\
\end{abstract}
	
	\maketitle
	
\par
\medskip\noindent
{\sc MSC (2020): 35D30, 35Q31, 35Q49.
	\par
}

\section{Introduction}

Let $\Omega\subset \R^2$ be an open and connected set.\ In this paper we consider $L^2(\Omega,\R^2)$ solutions of the system:
\begin{equation}\label{sys0}
	\begin{cases}
		\dv(u\otimes u) = 0,  &\text{ in }\Omega,\\
		\dv(u) = 0,  &\text{ in }\Omega.\\
	\end{cases}
\end{equation}
This system is solved by the simplest nontrivial solutions to the incompressible Euler equations:
\begin{equation}\label{eq:Eul}
	\partial_t u + \dv(u\otimes u) + Dp = 0, \quad \dv u  = 0,
\end{equation}
namely \emph{shear flows}, i.e. solutions $u$ of the form $u(x) = g((x,a))a^\perp$, for some fixed vector $a \in \R^2$ and a sufficiently regular profile $g: \R \to \R$.\ The question we study in this work is under which regularity solving \eqref{sys0} implies that $u$ is a shear flow.\ If this happens, we will say that system \eqref{sys0} is rigid.
\\
\\
The topic of rigidity and flexibility, intended in a broad sense, for solutions to the incompressible Euler equations has received considerable attention in the literature in recent years.\ The line of research which is closest to the topics of this paper is represented by \cite{Hamel2016,Hamel2019,Hamel2021,Drivas2024,Constantin2021,Elgindi2024,Gui2024,Ruiz2023,Wang2023,Enciso2024}.\ In these papers, the common theme is to study under which properties on the domain and on the time independent solution to \eqref{eq:Eul} one obtains additional symmetries or even stronger rigidity properties of the solution.\ As an example, one of the main results of the first of the mentioned papers \cite[Theorem 1.1]{Hamel2016} shows that if $u = (u_1,u_2)$ is a steady state of the planar Euler equations which is $C^2$ on a vertical strip, $u_2 = 0$ on the boundary of the strip, and $|u|$ is bounded below by a positive constant, then $u$ must be a shear flow.\ For other results on steady states, see also \cite{Choffrut2012,Enciso2025,Constantin2019,Gavrilov2019,Drivas2023}.\ 

Another setting where rigidity and flexibility has been studied are the deep convex integration works \cite{LINF,CONT,DECIMO,QUINTO,IS} that, together with \cite{CET}, settled Onsager's conjecture.\ On the topic, see also \cite{Novack2023, Giri2024, Giri2023,Brue2024,DRTE,Choffrut2014} and references therein.\ While these works are further from the topic of this paper, the convex integration techniques we use here are still quite close to \cite{LINF,Choffrut2014}, and the (conjectural) threshold for rigidity $p = 3$ that we will mention in Section \ref{s:es} is related, as in Onsager's conjecture, to guaranteeing that a trilinear term involving the solution $u$ is well-behaved. 
\\
\\
The analysis of \eqref{sys0} is rather different than the one of \eqref{eq:Eul} and is in fact substantially simpler.\ The main reason is that the absence of the time derivative and the pressure terms turn \eqref{eq:Eul} into a pure transport equation \eqref{sys0}, in which $u$ is transported by its own flow, at least if $u$ is smooth.\ This fact is also at the essence of the rigidity results we show.\ In parallel to the results we present here, one may ask what happens in the time-dependent, pressureless version of \eqref{eq:Eul}.\ T. Drivas and P. Isett \cite{ID} informed the author that, if one looks at $C^\alpha_{t,x}$ solutions to the pressureless version of \eqref{eq:Eul}, for $\alpha$ sufficiently small one can show flexibility of the system via convex integration constructions, while if $\alpha > \frac{1}{2}$ solutions are uniquely determined by their initial conditions.\ In addition, if $\alpha > \frac{2}{3}$, then $u$ must necessarily be a shear flow.\ Rather interestingly, while the time-dependent case requires solutions to be at least $C^\alpha$, the correct scale to study problem \eqref{sys0} is the one of far less regular $L_{\loc}^p$ spaces.\ The reason is that in order to study the pressureless version of \eqref{eq:Eul}, one needs to take the divergence of the Euler equations and work with the equation
\[
\dv(\dv (u\otimes u)) = 0,
\]
which is a very weak form of the Monge-Amp\`ere equation, for which the correct setting is the $C^\alpha$ scale, see for instance \cite{L1,Cao2019,Cao2025a} and references therein.\ To explain why in the case of interest for this work $u \in L^p_{\loc}$ is sufficient to hope for rigidity, let us now move to explain in detail the results of this paper.
\\
\\
First of all, notice that system \eqref{sys0} needs to be considered in the sense of distributions as:
\begin{align}
&\int_{\Omega}(u,DX u)\;dx = 0, \quad \forall X \in C^\infty_c(\Omega,\R^{2\times 2}), \label{weakuou}\\
&\int_{\Omega}(u,D\eta)\;dx = 0, \quad \forall \eta \in C^\infty_c(\Omega), \label{weakdiv}
\end{align}
respectively. If $\Omega \neq \R^2$, we add some boundary conditions, the most natural being:
\begin{equation}\label{sys}
	\begin{cases}
	\dv(u\otimes u) = 0,\;\dv(u) = 0,  \text{ in }\Omega,\\
	u\cdot n = 0, \text{ on }\partial\Omega,
	\end{cases}
\end{equation}
if $n$ represents the (outer) unit normal to $\Omega$.\ This means that $u$ still solves \eqref{weakuou}, while \eqref{weakdiv} is replaced by 
\begin{equation}\label{weak}
	\int_{\Omega}(u,D\eta)\;dx = 0,\quad \forall \eta \in C^\infty_c(\R^2).
\end{equation}
Our results concern both exact and approximate solutions.

\subsection{Exact solutions}\label{s:es}
Our main rigidity result is the following:
\begin{thmx}[Rigidity for exact solutions]\label{t:b}
	Let $\Omega \subset \R^2$ be a domain such that $(\overline{\Omega})^c$ has finitely many connected components and $\partial \Omega = \partial((\overline{\Omega})^c)$.\ Then:
	\begin{enumerate}[(i)] 
		\item\label{st:1} if $u \in L^2_{\loc}(\Omega,\R^2)$ is a solution to system \eqref{sys} with $\chi_\Omega u \in L^3_{\loc}(\R^2,\R^2)$, then $\chi_\Omega u$ solves \eqref{sys0} in $\R^2$;
		\item\label{st:2} if $\Omega = \R^2$ and $u \in L^3_{\loc}(\R^2,\R^2)$ then  $u$ is a shear flow; 
		\item\label{st:3} if $\Omega = \R^2$ and $u \in L^p(\R^2,\R^2)$ for some $p \in [2,\infty)$, then $u \equiv 0$.
	\end{enumerate} 
\end{thmx}
 
The first statement \eqref{st:1} asserts that, under the boundary conditions of \eqref{sys}, we may reduce to the study of the global case, which is considered in \eqref{st:2}.\ \eqref{st:3} tells us that if instead of $u \in L^p_{\loc}$ we assume the stronger requirement that $u \in L^p(\R^2,\R^2)$, then $u$ needs to vanish identically, with no constraint on the value of $p \ge 2$.\ The proofs of \eqref{st:2}-\eqref{st:3} use as main ingredients a combination of the ideas of \cite{KIRK,Pakzad2004} and of \cite{Alberti2014,Bianchini2016}.\ In particular, we need the chain rule property of transport equations with vector-field in $L^p$, which we show in Appendix \ref{ABCp} using the methods of \cite{Alberti2014,Bianchini2016}.\ A similar result in a harder case was obtained in \cite{GUS}.\ Interestingly, we also obtain that non-vanishing solutions to \eqref{sys} are rigid, without further boundary conditions:

\begin{thmx}[Rigidity for non-vanishing exact solutions]\label{t:c}
	Let $\Omega$ be a domain.\ Let $u \in L_{\loc}^3(\Omega,\R^2)$ be a solution to system \eqref{sys0} with $u \neq 0$ a.e.. Then, $u$ is a shear flow. If $u$ is allowed to be zero on sets of positive measure, then there exists an example of non-shear flow solving \eqref{sys0}.
\end{thmx}

Concerning flexibility of exact solutions, namely the sharpness of the exponent $p = 3$ in Theorem \ref{t:b}, we leave the following:

\begin{conx}\label{conj}
	If $p < 3$, there exists a non-zero $L^p(B_1,\R^2)$ solution to \eqref{sys} in $\Omega = B_1$.
\end{conx}

We believe the question should have an affirmative answer, and a possible way to show it is to use the convex integration methods of Section \ref{sec:te}. However, we were not able to complete the construction, and hence we leave it as an open question.

\subsection{Approximate solutions}

Alongside \eqref{sys0}-\eqref{sys}, we also consider \emph{approximate} solutions of it, i.e.\, given $p \ge 2$, we study limits of sequences $(u_j)_j$ weakly converging to $u \in L^p(B_1,\R^2)$ solving
\begin{equation}\label{e:app}
	\begin{cases}
		\dv(u_j \otimes u_j) = \dv(X_j),\\
		\dv(u_j) = 0,\\
		\sup_j\|X_j\|_{L^{\frac{p}{2}}} < +\infty,\; X_j \to 0 \text{ in }L^1(\Omega).
	\end{cases}
\end{equation} 
If $X_j = p_j\id$, we can interpret $u_j$ as solutions to stationary incompressible Euler equations with \emph{vanishing} pressure.\ The classical tool to study problem \eqref{e:app} is provided by Young measures, that we briefly recall in Subsection \ref{sec:YM}, and we refer the reader to \cite[Section 3]{DMU} for a more complete account on the subject.\ For the moment, we will trade formalism for clarity in exposition, and we will simply assume that the sequence
\begin{equation}\label{e:AAj}
M_j = \left(\begin{array}{c}u_j \\ u_j\otimes u_j - X_j \end{array}\right) \text{generates the Young Measure } \nu = (\nu_x)_x.
\end{equation}
Such probability measures $\nu_x \in \mathcal{P}(\R^{3\times 2})$ are supported in the inclusion set:
\begin{equation}\label{e:diffinc}
	\spt(\nu_x) \subset K\doteq \left\{\left(\begin{array}{cc}y \\ y\otimes y\end{array}\right): y \in \R^2\right\}.
\end{equation}

The same rigidity question raised in the previous section about exact solutions can be raised here in the case of approximate solutions. This translates in additional properties of the support of the Young measure $\nu$. For a shear flow $u(x) = g((x,a))a^\perp$, we have
\[
\left(\begin{array}{cc}u(x) \\ u(x)\otimes u(x)\end{array}\right) = \left(\begin{array}{cc}g((x,a))a \\ g((x,a))^2a\otimes a\end{array}\right), \text{ i.e. } \left(\begin{array}{cc}u(x) \\ u(x)\otimes u(x)\end{array}\right) \in K_a = \left\{\left(\begin{array}{cc}ta \\ t^2a\otimes a\end{array}\right): t \in \R \right\} \text{ a.e.}.
\]
This is the same rigidity we expect from gradient Young measures:

\begin{thmx}[Approximate rigidity for $p > 4$]\label{t:d}
	Let $p > 4$ and $(u_j)_j$, $(X_j)_j$ be sequences of solutions to \eqref{e:app} with $\sup_j\|u_j\|_{L^p(B_1,\R^2)}+ \sup_j\|X_j\|_{L^{p/2}(B_1,\R^{2\times 2})} < + \infty$.\ 	Assume $(M_j)_j$ defined in \eqref{e:AAj} generates the Young measure $\nu = (\nu_x)_{x \in B_1}$.\ Then, there exists $b(x) \in \R^2$ such that $\spt(\nu_x) \subset K_{b(x)}$ for a.e. $x \in B_1$.
\end{thmx}

The previous theorem holds if $p > 4$, and this is sharp: 

\begin{thmx}[Approximate flexibility for $p < 4$]\label{t:e}
	There exist sequences of solutions $(u_j)_j$, $(X_j)_j$ to \eqref{e:app} with $\sup_j\|u_j\|_{L^p(B_1,\R^2)} + \sup_j\|X_j\|_{L^{p/2}(B_1,\R^{2\times 2})} < + \infty$ for all $p < 4$ such that the sequence $(M_j)_j$ generates a (homogeneous) Young measure $\nu$ with $\spt{\nu} = K_{e_1}\cup K_{e_2}$, where $e_1$ and $e_2$ form the canonical basis of $\R^2$.
\end{thmx}

Combining Theorems \ref{t:b}-\ref{t:e} we obtain some interesting remarks.\ Theorem \ref{t:e} is shown by constructing a \emph{staircase laminate}, a convex integration technique introduced by D. Faraco in \cite{MIL}, and since then used in many works \cite{AFSZ,FMCO,CFM,CFMM,CT,JOHA}.\ This method is very useful to construct pathological solutions to PDEs exhibiting \emph{concentration phenomena}.\ It typically consists of three steps:
\begin{enumerate}
	\item rewriting the system of PDEs at hand as a differential inclusion of the form $Dv \in \tilde K \subset \R^{n\times m}$;\label{i1}
	\item finding a staircase laminate with the required integrability properties supported in the inclusion set $\tilde K$ to obtain a sequence of approximate solutions.\ This generally introduces errors and one does not obtain an exact solution in this step;\label{i2}
	\item finding an \emph{in-approximation} to $\tilde K$ to eliminate such errors and eventually construct an exact solution with the required (pathological) properties.\label{i3}
\end{enumerate}

We refer the reader to Section \ref{sec:te} for details on the terminology we use. Combining Theorems \ref{t:b}-\ref{t:e} we find that $L^p_{\loc}$ exact solutions to \eqref{sys0} are rigid for $p \ge 3$, while approximate solutions need not be in the range $p \in [3,4)$. In particular in this range Step \eqref{i3} fails, and thus this problem provides an interesting example of differential inclusion for which approximate flexibility holds while no in-approximation is available. This can be explained by noticing that the methods used to show Theorem \ref{t:b} fail to treat right-hand sides belonging to a negative Sobolev space, see \eqref{e:app} (while they work for measure right-hand sides, see \cite{Bianchini2016}). In particular, a very general method of solving Step \eqref{i3} once Step \eqref{i2} is achieved was recently introduced in \cite{Kleiner2024}, see also the very recent extension \cite{Buchowiec2026}.\ Roughly, the idea of \cite{Kleiner2024} is to show that \eqref{i3} can always be achieved provided one can find enough laminates as in Step \eqref{i2} to \emph{reduce the inclusion set to $\R^{n\times m}$}.\ The set $K$ we consider here, see \eqref{e:diffinc}, is an example of a set that cannot be reduced to $\R^{3\times 2}$ if one looks at $L^p$ solutions for $p \in [3,4)$.

\subsection{Structure of the paper}

We start by explaining the notation we use in this paper in Section \ref{sec:not}.\ Next, in Section \ref{sec:ta} we show Theorem \ref{t:b} and in Section \ref{subs:rig} we show Theorem \ref{t:c}.\ The results on Young measures and approximate solutions, namely Theorems \ref{t:d} and \ref{t:e}, are contained in Sections \ref{sec:apprig} and \ref{sec:te} respectively.\ In Appendix \ref{ABCp}, we show a version of \cite{Alberti2014} valid for $L^p$ vector fields.

\subsection*{Aknowledgements}
The author thanks T. Drivas and P. Isett for posing this question and discussing with him the literature, their results on the time-dependent version of the equation considered here and his manuscript.\ He also wishes to thank A. Guerra and L. De Rosa for reading parts of this manuscript and for useful discussions about it. 

%

\subsection{Notation}\label{sec:not}
 
For $E \subset \R^n$ we denote by $\overline{E}$ its closure, by $\partial E$ its topological boundary, and by $E^c$ its complement in $\R^n$.\ We call $E$ a \emph{domain} if $E$ is open and connected.\ For two sets $A,B$, we denote by $\dist(A,B)$ the distance between them.\ The open ball of radius $r$ centered at $Y$ is denoted by $B_r(Y)$.\ If $Y = 0$, we will simply write $B_r$.\

We let $\mathcal{M}(\R^m)$ be the space of finite and positive measures on $\R^m$ and $\mathcal{P}(\R^m)$ be the space of probability measure.\ $\mathcal{L}^m$ denotes the Lebesgue measure on $\R^m$ and for each $\mu \in \mathcal{M}(\R^{m})$ and $f \in L^1(\R^m;\mu)$, we set
\[
\langle\mu,f\rangle \doteq \int_{\R^{m}}f(y)d\mu(y).
\] 
 
We write $e_i$ for the $i$-th vector of the canonical basis of $\R^2$, $e_1 = (1,0)$ and $e_2 = (0,1)$. For a matrix $A \in \R^{n\times m}$, $\det(A)$, $A^T$ and $|A|$ denote the determinant (if $n = m$), the transpose and the Euclidean norm of the matrix $A$, respectively.\ The (standard) scalar product between matrices is denoted by $\langle A,B\rangle$, while for vectors $a,b$ we use $(a,b)$.\ The cofactor matrix is
\begin{equation}\label{eq:A}
	\cof(A) \doteq \left(\begin{array}{cc} d & -b \\ -c & a\end{array}\right), \quad \text{ if } A = \left(\begin{array}{cc} a & b \\ c & d\end{array}\right),\quad
	\text{so that }\cof(A)A = A\cof(A)= \det(A)\id\,.
\end{equation}

$\Sym(n)$ denotes the space of symmetric matrices of size $n$, and $\Sym^+(n)$ is the set of matrices $X \in \Sym(n)$ such that $(Xv,v) \ge 0$ for all $v \in \R^n$.\ We denote by $J$ the matrix representing a rotation of ninety degrees:
\begin{equation}\label{e:J}
	J \doteq \left(\begin{array}{cc}0 & 1 \\ -1 & 0\end{array}\right).
\end{equation}
As a shorthand notation, we set $v^\perp \doteq Jv, \forall v \in \R^2$.\ Finally, for a matrix $X \in \R^{3\times 2}$, we denote by $\pi_1(X)$ the projection on the first row of $X$ and by $\pi_2(X)$ the projection on the last two rows of $X$.

	\section{Proof of Theorem \ref{t:b}}\label{sec:ta}
	
	This section is divided into the three Subsections \ref{sec:st:1}-\ref{sec:st:2}-\ref{sec:st:3} in which we show Theorem \ref{t:b}\eqref{st:1}-\eqref{st:2}-\eqref{st:3} respectively.
	
	\subsection{Proof of Theorem \ref{t:b}\eqref{st:1}: reduction to the global case}\label{sec:st:1}
	We need to show that 
	\begin{equation}\label{eq:distrib}
	\dv(\chi_\Omega u\otimes u) = 0 \text{ in the sense of distributions in $\R^2$}.
	\end{equation}
 
	To this end, we let $v = \chi_\Omega u$. By \eqref{weak} this is a divergence-free, $L_{\loc}^3(\R^2,\R^2)$ vector field. Thus, we find 
		\begin{equation}\label{uv}
			v = D^\perp \psi \text{ a.e. in $\R^2$}, \quad \text{for $\psi \in W_{\loc}^{1,3}(\R^2)$.}
		\end{equation}
 
		Note that $\psi$ is locally constant outside $\overline{\Omega}$. As, by assumption, $(\overline{\Omega})^c = \Omega_1\cup\dots\cup \Omega_N$ where $\Omega_i$ is open and connected, we have $\psi \equiv a_i$ on $\Omega_i$.\ Since $\psi$ is continuous and $\partial \Omega = \partial((\overline{\Omega})^c)$ we find $\psi(\partial\Omega) \subset \{a_1,\dots, a_N\}$.\ Hence, letting $\varphi \in C^\infty_c(\R^2,\R^2)$ and $f \in \Lip\cap L^\infty(\R)$ with $f(a_i) = 0$ for every $i$, we infer:
		\[
		X \doteq f(\psi) \varphi \in W^{1,3}_0(\Omega,\R^2).
		\]
		A straightforward approximation argument shows that $X$ is a valid test function for \eqref{weakuou}. We obtain:
		\[
		0 = \int_{\Omega}(u,DX u)\; dx = \int_{\Omega}(u,D(f(\psi) \varphi) u)\; dx = \int_{\Omega}f(\psi)(u,D \varphi u)\; dx + \int_{\Omega}f'(\psi)(u, [\varphi \otimes D\psi] u)\; dx.
		\]
		Due to \eqref{uv}, we find that
		\[
		(u, [\varphi \otimes D\psi] u) = (u,\varphi)(u,D\psi) = 0 \text{ a.e. in }\Omega.
		\]
		We just showed that, for $f \in \Lip(\R)$ with $f(a_i) = 0$ for all $i$,
		\begin{equation}\label{final}
			\int_{\Omega}f(\psi)(u,D \varphi u)\; dx = 0.
		\end{equation}
		Consider, for any $\eps > 0$,
		\[
		g_\eps(t) \doteq \frac{|t|}{\sqrt{t^2 + \eps}}.
		\]
		Then $g_\eps$ is a Lipschitz function on $\R$ with $g_\eps(0) = 0$. We further observe that
		\begin{equation}\label{a}
			|g_\eps(t)| \le 1 \text{ for all }t \in \R \text{ and } g_{\eps}(t) \to \chi_{\{s \neq 0\}}(t) \text{ a.e.}.
		\end{equation}
		Using $f_\eps(t) \doteq \prod_ig_\eps(t-a_i)$ instead of $f(t)$ in \eqref{final}, \eqref{a} and the dominated convergence theorem yield:
		\[
		\int_{\{\psi \neq a_1,\dots,a_N\}}(u,D \varphi u)\; dx = 0.
		\]
		Since, for any $a$, $u = 0$ a.e.\ on the set $\{\psi = a\}$, the previous equality implies the required property \eqref{eq:distrib}. \qed

\begin{remark}
We do not know if the (quite mild) topological assumptions of Theorem \ref{t:b}\eqref{st:1} can be relaxed.\ Notice that similar computations appeared in \cite[Section 2.4]{DelNin2024}, where the assumption was similar to ours.
\end{remark}

	\subsection{Preliminaries and proof of Theorem \ref{t:b}\eqref{st:2}}\label{sec:st:2}
	
	As said in the introduction, the statement follows from a combination of the results of \cite{KIRK, Pakzad2004} and \cite{Alberti2014,Bianchini2016}. In particular, we need the following version of \cite{Alberti2014}:
	
	\begin{theorem}\label{chain}
		Let $\Omega \subset \R^2$ be an open set. Let $\beta_i \in L_{\loc}^r(\Omega)$ for $i =1,\dots, N$ and $u \in L_{\loc}^p(\Omega)$, for $r,p \in [1,+\infty]$. Assume
		\begin{equation}\label{bounds}
			\frac{1}{r} + \frac{2}{p} \le 1
		\end{equation}
		and, in the sense of distributions in $\Omega$:
		\begin{align}
			\dv(u) &= 0,\label{p1}\\
			\dv(\beta_i u) &= 0,\quad \forall i.\label{p2}
		\end{align}
		Then, for all $f \in C^0(\R^N)$ with $|f(x)| \le C(1 + |x|) \text{ for some $C > 0$ for all $x \in \R^N$}$, it holds
		\begin{equation}\label{cons}
			\dv(f(\beta_1,\dots, \beta_N) u) = 0.
		\end{equation}
	\end{theorem}
	
	A proof of this theorem will be given in the Appendix, Section \ref{ABCp}.\ For the convenience of the reader, we also recall the results of \cite{KIRK, Pakzad2004}:
	
	\begin{theorem}\label{kirp}
		Let $\Omega \subset \R^2$ be open.\ Let $f \in W^{1,2}_{\loc}(\Omega,\R^2)$ be a map with $Df \in \{X \in \Sym(2):\det(X) = 0\}$ for a.e. $x \in \Omega$. Then, for every $x \in \Omega$, there exists either a neighborhood $U$ of $x$, or a segment passing through $x$ and joining $\partial \Omega$ at its both ends, on which $f$ is constant.\ In particular, if $\Omega = \R^2$, then there exist $a \in \R^2$ and $g \in W^{1,2}_{\loc}(\R)$ such that $f(x) = g((x,a))a.$
	\end{theorem}
	The first part of the statement can be found in \cite[Proposition 1.1]{Pakzad2004}, while the part concerning the case $\Omega = \R^2$ can be deduced from \cite[Corollary 2.31]{KIRK}.\ We can now show the rigidity claimed in Theorem \ref{t:b}\eqref{st:2}.
	
	\begin{proof}[Proof of Theorem \ref{t:b}\eqref{st:2}]
	Employing System \eqref{sys0}, which holds in $\R^2$, we obtain:
		\begin{equation}\label{ren}
			\dv\left(\frac{u\otimes u}{1+|u|^2}\right) = 0 \text{ in the sense of distributions in $\R^2$}.
		\end{equation}
		This follows from Theorem \ref{chain} applied with $\beta_i = u_i$, if $u = (u_1,u_2)$, and $f(z)\doteq (1 + |z|^2)^{-1}$.\ Now \eqref{ren} yields, through Poincar\'e Lemma,
		\begin{equation}\label{uuu}
		\frac{u\otimes u}{1+|u|^2} = \cof(D^2\varphi).
		\end{equation}
		We thus have $\varphi \in W_{\loc}^{2,\infty}(\R^2)$ with $\det(D^2\varphi) = 0$ a.e.. We can then employ Theorem \ref{kirp} with $f = D\varphi$ to infer that $D\varphi(x,y)= g((x,a))a$ for some $g \in W^{1,\infty}(\R)$ and $a \in \R^2$.\ If $a = 0$, then clearly $u \equiv 0$ and the proof is concluded.\ If $a \neq 0$, from \eqref{uuu} we infer
		\[
		\frac{u\otimes u}{1+|u|^2} = g'((x,a))a^\perp\otimes a^\perp.
		\]
		Multiplying by $a$ both sides, we finally get that $u(x) = \lambda(x)a^\perp$.\ Since $\dv u = 0$, then $u$ is a shear flow.
 
		\end{proof}

\subsection{Proof of Theorem \ref{t:b}\eqref{st:3}: unconstrained rigidity for $L^p$ vector fields}\label{sec:st:3}

Let us start with the following:

\begin{lemma}\label{lem:divfree}
	Let $A \in L^1_{\loc}(\R^n,\Sym^+(n))$ with $\dv A = 0$.\ Then, for all $x \in \R^n$, the function
	\begin{equation}\label{eq:fmono}
	f(R)\doteq \frac{1}{R}\int_{B_R(x)}\tr(A)(y)dy
	\end{equation}
	is monotone increasing in $R$.\ As a consequence, if $A \in L^p(\R^n,\Sym^+(n))$, for $p < \frac{n}{n - 1}$, then $A \equiv 0$.
\end{lemma}
	\begin{proof}
		We can assume without loss of generality that $A$ is smooth and $x = 0$.\ To show that $f$ is monotone, we simply differentiate \eqref{eq:fmono}
		\[
		f'(R) =\frac{1}{R}\left[-\frac{1}{R}\int_{B_R}\tr(A)dx + \int_{\partial B_R}\tr(A)dx\right]
		\]
		and we rewrite using $\dv A = 0$:
		\[
		\int_{B_R}\tr(A)dx = \int_{B_R}\dv(Ax)dx = R\int_{\partial B_R}\left(A\frac{x}{|x|},\frac{x}{|x|}\right)dx \le R\int_{\partial B_R}\tr(A)dx,
		\]
		so that $f'(R) \ge 0$ for all $R>0$ follows.\ Now if $A \in L^q(\R^n,\Sym^+(n))$ and $q < \frac{n}{n - 1}$, then $\lim_{R \to \infty}f(R) = 0$ by H\"older inequality, and hence $A \equiv 0$.
	\end{proof}
	We can finally conclude the proof of Theorem \ref{t:b}.\
\begin{proof}[Proof of Theorem \ref{t:b}\eqref{st:3}]
	Let first $p \ge 3$. We can then use Theorem \ref{t:b}\eqref{st:2} to deduce that $u$ is a shear flow and hence $0$, since $u \in L^p(\R^2,\R^2), p < + \infty$.\ We can then assume $p < 3$, where the proof is concluded by applying Lemma \ref{lem:divfree} with $A \doteq u \otimes u \in L^\frac{3}{2}(\R^2,\Sym^+(2))$.
\end{proof}

\section{Proof of Theorem \ref{t:c}: rigidity for non-vanishing vector fields}\label{subs:rig}
As the theorem is local in nature, we can assume $\Omega = B_1$.\ Similarly to the proof of Theorem \ref{t:b}\eqref{st:2}, applying Theorem \ref{chain} we find that
\begin{equation*}\label{ren1}
	\dv\left(\frac{u\otimes u}{\eps+|u|^2}\right) = 0,\quad  \text{ for all $\eps > 0$}.
\end{equation*}
Letting $\eps \to 0$ and using that $u \neq 0$ a.e. in $\Omega$, we obtain
\begin{equation}\label{ren2}
	\dv\left(\frac{u\otimes u}{|u|^2}\right) = 0.
\end{equation}
As in Theorem \ref{chain} and since $B_1$ is convex, we can then find $\varphi \in W^{2,\infty}$ with
\[
\cof(D^2\varphi) = \frac{u\otimes u}{|u|^2}.
\]
In terms of a differential inclusion, the last equation is equivalent to
$$D^2\varphi \in \Gamma = \{X \in \Sym^+(2): \det(X) = 0, |X| = 1\} = \{x\otimes x: x \in \mathbb{S}^1\}, \text{ a.e. in $B_1$}.$$
The set $\Gamma$ is an elliptic curve, i.e. it is the image of the curve $\Gamma(t)\doteq e^{it}\otimes e^{it}$ which is \emph{elliptic} in the sense that there exists $C < 0$ such that
\[
C\det(\Gamma(t)-\Gamma(s)) \ge |\Gamma(t)-\Gamma(s)|^2, \quad \forall t,s \in \R,
\]
as can be easily checked by direct computations.\ We can then employ \cite[Lemma 2.1]{Lamy2023} to infer that $\frac{u\otimes u}{|u|^2}$ is constant, which in turn implies that $u$ is a shear flow as at the end of the proof of Theorem \ref{t:b}\eqref{st:2}.
\\
\\
It remains to show that the assumption $u \neq 0$ a.e. is optimal.\ Indeed, a counterexample to the first part of Theorem \ref{t:c} can be easily found by considering $\Omega = (0,1)\times \R$ and
\begin{equation}
	u(x) = \begin{cases}
		e_1, &\text{ if } y \le 0,\\
		e_1 + e_2, &\text{ if } y \ge x,\\
		0, &\text{ if } 0 \le y \le x.
	\end{cases}
\end{equation}\qed

\section{Young measures and proof of Theorem \ref{t:d}}\label{sec:apprig}

In order to show Theorems \ref{t:d}-\ref{t:e}, we will first properly introduce gradient Young measures.\ Recall, from now on, that the set $K$ is the one defined in \eqref{e:diffinc}.
\subsection{Young measures}\label{sec:YM}
Let $p \ge 2$ and consider \eqref{e:app}, for an equibounded sequence $(u_j)_j \subset L^p(B_1,\R^2)$:
\begin{equation}\label{e:app1}
	\begin{cases}
		\dv(u_j \otimes u_j) = \dv(X_j),\\
		\dv(u_j) = 0,\\
		\sup_j\|X_j\|_{L^{\frac{p}{2}}} < +\infty,\; X_j \to 0 \text{ in }L^1(\Omega).
	\end{cases}
\end{equation}
It is more convenient to consider the \emph{curl} version of this system, which is obtained via a simple multiplication by a rotation $J$ of ninety degrees in the plane, see \eqref{e:J}. Consider in fact the sequence
\begin{equation}\label{e:Aj}
	A_j(x) \doteq \left(\begin{array}{cc}Ju_j\\ Ju_j\otimes Ju_j + JX_jJ\end{array}\right),
\end{equation}
which satisfies $\curl A_j = 0$. Therefore, $A_j = D\psi_j$ a.e. for some potential $\psi_j: B_1 \subset \R^2 \to \R^3$.\ We want to say that $(D\psi_j)_j$ generates a gradient Young measure.\

\begin{theorem}[Fundamental Theorem on Young measures]\label{FUN}
	Let $E\subset \R^n$ be a Lebesgue measurable set with finite measure and let $p > 1$.\ Consider a sequence $(z_j)_j \subset L^p(E,\R^m)$ weakly converging in $L^p$ to some function $z$.\ Then, there exists a subsequence $(z_{j_k})_k$ that \emph{generates the Young measure} $\nu$, where $\nu$ is a weak-$*$ measurable map $\nu:E \to \mathcal{M}(\R^m)$ such that for $\mathcal{L}^n$-a.e. $x \in E$, $\nu_x \in \mathcal{P}(\R^m)$.\ The family $\nu = (\nu_x)_{x \in E}$ has the property that for every $f \in C(\R^m)$ such that
	\[
	|f(y)|\le C(1 + |y|^q), \text{ for } q < p \text{ and for some $C > 0$},
	\]
	the following holds
	\[
	f(z_j) \rightharpoonup \bar{f}, \text{ weakly in } L^{\frac{p}{q}}(E), \quad \text{ where }\bar f(x) = \langle \nu_x,f \rangle.
	\]
	In particular, the choice $f(y) = y,\; \forall y \in \R^m$ yields
	\begin{equation}\label{exp}
		z(x) = \langle\nu_x,f\rangle.
	\end{equation}
\end{theorem}
\noindent If the sequence $z_j = D\psi_j$, as in the case above, then we will call $\nu$ a \emph{gradient Young measure}.
\\
\\
We identify the Young measure generated by the matrix fields $M_j$ of \eqref{e:AAj} and the one generated by the matrix field $A_j$ of \eqref{e:Aj}, since anyway they would only differ by a linear transformation.\ We collect these considerations and a useful restating of the properties of the matrix fields $(A_j)_j$ in the next:

\begin{lemma}\label{lem:tec}
Let $p > 2$. Given $(u_j)_j$ equibounded in $L^p(B_1,\R^2)$ solving \eqref{e:app1} for some sequence $(X_j)_j$, let $A_j$ be the matrix field defined in \eqref{e:Aj}. Then, $A_j = D\psi_j$, for $\psi_j: B_1 \subset \R^2 \to \R^3$ with zero average and, if $\psi_j = (f_j,F_j)$, $f_j: B_1 \subset \R^2 \to \R$ and $F_j: B_1\subset \R^2 \to \R^2$, we have
\begin{equation}\label{e:gro}
	\|f_j\|_{W^{1,p}(B_1)} + \|F_j\|_{W^{1,\frac{p}{2}}(B_1,\R^2)} \le C, \quad \forall j.
\end{equation}
Furthermore, for any Young measure $(\nu_x)_x$ generated by $(A_j)_j$, $\nu_x$ is supported in $K$ for a.e. $x$, where $K$ is defined in \eqref{e:diffinc}. Conversely, given a sequence $\varphi_j: B_1 \subset \R^2 \to \R^3$ with $\varphi_j = (g_j,G_j)$ fulfilling \eqref{e:gro} and
\[
\dist(D\varphi_j,K) \to 0 \text{ in }L^1(B_1),
\]
there exist $u_j,X_j$ which solve \eqref{e:app1}, with $(u_j)_j$ equibounded in $L^p(B_1,\R^2)$.\ Moreover, given such $u_j,X_j$ and defined $(A_j)_j$ as in \eqref{e:Aj}, we have $D\varphi_j = A_j$ and hence if (a subsequence of) $(D\varphi_j)_j$ generates the Young measure $\nu = (\nu_x)_{x \in B_1}$, then the corresponding subsequence of $(A_j)_j$ also generates $\nu$.
\end{lemma}
\begin{proof}
	The first part is easier.\ In particular, \eqref{e:gro} is trivial. Let us consider a Young measure $(\nu_x)_x$ generated by $(A_j)_j$.\ By observing that, by definition \eqref{e:Aj} of $(A_j)_j$,
	\[
	\dist(D\psi_j,K) \le |X_j|,
	\]
	we have that $\dist(D\psi_j,K) \to 0$ in $L^1$. Therefore, for any smooth $\eta \in C^\infty_c(B_1)$, we can use the definition of Young measures to infer
	\[
	0 = \lim_j \int_{B_1}\dist(D\psi_j,K)(x)\eta(x)dx = \int_{B_1}\eta(x)\int_{\R^{3\times 2}}\dist(X,K)d\nu_x dx.
	\]
	We thus find that, for a.e. $x \in B_1$, $\dist(X,K) = 0$ for $\nu_x$-a.e. $X$. Let us now move to the second part of the proof. We define $u_j \doteq -JDg_j$, so that $\dv(u_j) = 0$ and $(u_j)_j$ is equibounded in $L^p$.\ Next let
	\begin{equation}\label{eq:XuF}
	X_j \doteq u_j\otimes u_j-\cof^T(DG_j).
	\end{equation}
	Observe that if $(A_j)_j$ are the matrix fields associated to these $u_j$, $X_j$ as in \eqref{e:Aj}, then:
	\begin{equation}\label{e:ajD}
	A_j = D\varphi_j \text{ for all } j.
	\end{equation}
	From \eqref{eq:XuF}, we immediately find that $(X_j)_j$ is equibounded in $L^{\frac{p}{2}}$ and
	\[
	\dv(u_j\otimes u_j) = \dv(X_j).
	\]
	We thus only need to show that $X_j \to 0$ in $L^1$. Equivalently, we only need to show that its norm converges to zero in measure, since it is equibounded in $L^\frac{p}{2}$ and $p > 2$.\ To see this, consider $g: \R^{3\times 2} \to \R$ defined as
	\[
	g\left(X\right) \doteq |a\otimes a - A|, \quad \text{ for } X = \left(\begin{array}{c}a\\ A\end{array}\right) \in \R^{3\times 2}, a\in \R^2,A \in \R^{2\times 2}.
	\]
	For any other matrix $Y = \left(\begin{array}{c}b\\B\end{array}\right)$, we have
	\begin{equation}\label{eq:g}
	|g\left(X\right) - g(Y)| \le (|a| +|b|)|a-b| + |A-B|.
	\end{equation}
	Let $Z = \left(\begin{array}{c}b\\ b\otimes b\end{array}\right) \in K$ be a matrix for which $\dist(X,Z) = \dist(X,K)$.\ From \eqref{eq:g}, we infer
	\[
	g(X) \le (1+3|X|)\dist(X,K).
	\]
	We evaluate this inequality at $D\varphi_j$, and we notice that $|X_j| = g(D\varphi_j)$:
	\[
	|X_j| = g(D\varphi_j) \le (1+ 3|D\varphi_j|)\dist(D\varphi_j,K), \quad \text{a.e. in }B_1.
	\]
	Exploiting Lemma \ref{lem:basic} below, we deduce that $|X_j|$ converges to $0$ in measure, and hence that $X_j\to 0$ in $L^1(B_1,\R^{2\times 2})$.\ 
\end{proof}

\begin{lemma}\label{lem:basic}
Let $(a_j)_j$ and $(b_j)_j$ be sequences such that $\|a_j\|_{L^1(B_1)} \le C$ for all $j$ and $b_j \to 0$ in measure. Then, $p_j \doteq a_jb_j$ converges to zero in measure.
\end{lemma}
\begin{proof}
We assume, without loss of generality, that $a_j,b_j \ge 0$.\ Fix any $t > 0$ and $\eps > 0$. Then, for any $s > 0$:
	\begin{align*}
		|\{x: a_jb_j > t\}| &\le |\{x: a_jb_j > t\} \cap \{x: a_j \le s\}| + |\{x: a_jb_j > t\} \cap \{x: a_j \ge s\}|\\
		& \le |\{x: a_jb_j > t\} \cap \{x: a_j \le s\}| + |\{x: a_j \ge s\}|\\
		&  \le |\{x: b_j > t/s \}| + C/s,
	\end{align*}
the last inequality being true by Chebyshev's inequality. Now choose $s$ so that $C/s \le \eps/2$.\ Having fixed $s>0$, we can use the convergence in measure to $0$ of the sequence $(b_j)_j$ to conclude the proof.
\end{proof}

\subsection{Proof of Theorem \ref{t:d}: approximate rigidity}\label{subs:apprig}

Let $(u_j)_j, (X_j)_j$ be sequences of maps solving \eqref{e:app}, with $\|u_j\|_{L^p} \le C$ and $p > 4$. We wish to show that the gradient Young measure $(\nu_x)_{x \in B_1}$ generated by the matrix fields $(A_j)_j$ defined as in \eqref{e:Aj} is such that, for a.e. $x$, there exists $b(x) \in \R^2$ for which:
\begin{equation}\label{e:thesis}
\spt(\nu_x) \subset K_{b(x)}= \left\{\left(\begin{array}{cc}tb(x) \\ t^2b(x)\otimes b(x)\end{array}\right): t \in \R \right\}.
\end{equation}
We already know from Lemma \ref{lem:tec} that $\spt(\nu_x) \subset K$ for a.e.\ $x$, where $K$ is the set of \eqref{e:diffinc}.\ Let $A_j = D\psi_j$, with $\psi_j = (f_j,F_j)$, $f_j: B_1 \subset \R^2 \to \R$ and $F_j: B_1\subset \R^2 \to \R^2$ fulfilling the bounds \eqref{e:gro}.\ Using the definition of Young measure with the choice of function $h(X) \doteq X$ for all $X \in \R^{3\times 2}$, we find that $f_j$ and $F_j$ converge weakly in $L^p$ and $L^{p/2}$ to some functions $f$ and $F$, respectively.\ Furthermore, if $\psi = (f,F)$, then $D\psi(x)$ is the barycenter of $\nu_x$ for a.e. $x$, as expressed by \eqref{exp}.

Since $\spt(\nu_x) \subset K$, we obtain that also $DF(x) \in \Sym^+(2)$ for a.e. $x$. Furthermore, the determinant is a null Lagrangian, see \cite[Theorem 2.3]{DMU}, and hence for any $\varphi \in C^\infty_c(B_1)$,
\[
\lim_{j\to \infty}\int_{B_1}\varphi(x)\det(DF_j)(x)dx = \int_{B_1}\varphi(x)\det(DF)(x)dx.
\]
This readily implies that, for a.e. $x$,
\[
\det(DF(x)) = \int_{\R^{3\times 2}}\det(\pi_2(X))d\nu_x(X),
\]
recall the notation $\pi_i$ of Section \ref{sec:not}.\ As $\spt(\nu_x) \subset K$ for a.e. $x$, we deduce that $\det(DF(x)) = 0$ for a.e. $x$. In other words, for a.e. $x$ there exists a vector $b(x)$ such that
\[
DF(x)= b(x)\otimes b(x).
\] 
Letting $\mu_x'$ and $\mu_x$ be the pushforward of $\nu_x$ through $\pi_1(X)$ and $\pi_2(X)$ respectively, we now show that
\begin{equation}\label{e:mux}
\spt(\mu_x) \subset \{t^2b(x)\otimes b(x): t \in \R\}, \quad \text{ for a.e. }x \in B_1.
\end{equation}
This is a purely algebraic fact: as 
\[
\spt(\mu_x) \subset \{v\otimes v: v \in \R^2\}
\]
and its barycenter is the matrix $b(x)\otimes b(x)$, it suffices to consider the relation
\[
b(x)\otimes b(x) = \int_{\R^{2\times 2}}Xd\mu_x(X)
\]
and scalar multiply this with $Jb(x)\otimes Jb(x)$. This immediately implies \eqref{e:mux}. Up to now, we have shown that, for a.e. $x$,
\[
\spt(\nu_x) \subset \left\{\left(\begin{array}{cc}a \\ t^2b(x)\otimes b(x)\end{array}\right): a \in \R^2, t \in\R\right\} \cap K. 
\]
The only points $Z$ in the intersection of these two sets are precisely elements of $K_{b(x)}$. Indeed we need to have $a,b(x),u \in \R^2$ and $t \in \R$ such that
\[
Z = \left(\begin{array}{cc}a \\ t^2b(x)\otimes b(x)\end{array}\right) = \left(\begin{array}{cc}u \\ u\otimes u\end{array}\right).
\]
If $b(x) = 0$, then we immediately find $a = u = 0$. Thus, $Z = 0$, which belongs to $K_{b(x)} = \{0\}$. If $b(x) \neq 0$, then we can scalar multiply the last two rows by $Jb(x)\otimes Jb(x)$ to find $(u,Jb(x))^2 = 0$, and hence that $u = \lambda b(x)$ for some $\lambda \in \R$. Furthermore, $\lambda = \pm t$. Thus $Z$ reduces to
\[
Z = \left(\begin{array}{cc}\pm tb(x) \\ t^2b(x)\otimes b(x)\end{array}\right) \in K_{b(x)}.
\]
In other words, $\spt(\nu_x) \subset K_{b(x)}$ for a.e. $x$, and we conclude the proof.

\section{Staircase laminates and proof of Theorem \ref{t:e}}\label{sec:te}

In this section we will show Theorem \ref{t:e}.\ Specifically, recalling the sets
\[
K = \left\{X \in \R^{3\times 2}: X =\left(\begin{array}{cc} x \\ x\otimes x\end{array}\right)\right\} \text{ and } K_a = \left\{X \in \R^{3\times 2}: X =\left(\begin{array}{cc} ta \\ t^2a\otimes a\end{array}\right), t \in \R\right\},
\]
 we will construct a sequence $\psi_n = (g_n,G_n)$, for $g_n \in \Lip(B_1,\R)$ and $G_n \in \Lip(B_1,\R^2)$, $B_1\subset \R^2$, having the following properties:

\begin{enumerate}
	\item $\|g_n\|_{W^{1,p}(B_1)} + \|G_n\|_{W^{1,p/2}(B_1,\R^2)} \le C(p)$, for all $p < 4$;\label{item:conv}
	\item $\displaystyle\lim_n\int_{B_1}\dist(D\psi_n,K)dx = 0$;\label{item:dist}
\end{enumerate} 
In addition, the Young measure $(\nu_x)_{x \in B_1}$ generated by $(D\psi_n)_n$ is homogeneous, namely it does not depend on $x$ (we will therefore only denote it with $\nu \in \mathcal{P}(\R^{3\times 2})$), and enjoys the following properties
\begin{equation}\label{item:e}
\spt(\nu) \subset K_{e_1}\cup K_{e_2} \text{ and } \int_{\R^{3\times 2}}\dist(X,K_{e_i})d\nu > 0, \; \forall i = 1,2.
\end{equation}
Observe that this shows Theorem \ref{t:e} by means of the second part of Lemma \ref{lem:tec}. As said in the introduction, in this section we employ the convex integration method of \emph{staircase laminates}. We first recall this method and postpone the construction of the required laminate to the final subsection.

\subsection{Laminates of finite order and staircase laminates}

The basic idea is the following: let $A,B \in \R^{n\times m}$ be rank-one connected, namely $A-B = a\otimes b$, and let $\varphi$ be a function which equals $(1-\lambda) t$ on $[0,\lambda]$ and $-\lambda t$ on $[\lambda,1]$ and is extended to be $1$-periodic.\ Then the maps
\[
u_\eps(x) = (\lambda A + (1-\lambda) B)x + \eps\varphi((x,b)/\eps)a
\]
generate, as $\eps\to 0$, the Young measure
\[
\nu = \lambda \delta_{A} + (1-\lambda)\delta_{B}.
\]
This reasoning can be iterated, if for instance $B$ is the barycenter of a segment in another rank-one direction. The following definition explains this idea of splitting, and Proposition \ref{ind} provides the analog of the family $u_\eps$ after such iterations.

\begin{Def}\label{def:es}
	Let $\nu,\mu \in \mathcal{P}(\R^{n\times m})$. Let $\nu = \sum_{i= 1}^N\lambda_i\delta_{A_i}$. We say that $\mu$ can be obtained via \emph{elementary splitting from }$\nu$ if for some $i \in \{1,\dots,N\}$, there exist $B,C \in \R^{n\times m}$, $\lambda \in [0,1]$ such that
	\[
	\rank(B-C) = 1, \quad A_i = sB + (1-s)C,
	\]
	for some $s \in (0,1)$ and
	\[
	\mu = \nu +\lambda\lambda_i(-\delta_{A_i} + s\delta_B + (1-s)\delta_C).
	\]
	A measure $\nu = \sum_{i= 1}^r\lambda_i\delta_{A_i}\in \mathcal{P}(\R^{n\times m})$ is called a \emph{laminate of finite order} if there exists a finite number of measures $\nu_1,\dots,\nu_N \in \mathcal{P}(\R^{n\times m})$ such that
	\[
	\nu_1 = \delta_X, \nu_N = \nu
	\]
	and $\nu_{j + 1}$ can be obtained via elementary splitting from $\nu_j$, for every $j\in \{1,\dots,N-1\}$.
\end{Def}
The next Proposition can be found in \cite[Lemma 3.2]{SMVS}.
\begin{prop}\label{ind}
	Let $\nu = \sum_{i = 1}^{r}\lambda_i\delta_{A_i} \in \mathcal{P}(\R^{n\times m})$, $\lambda_i\neq 0, \forall i$, be a laminate of finite order, and let $A$ be the barycenter of $\nu$. If $\Omega$ is an open and bounded set, $b \in \R^m$ and $u(x)\doteq Ax+b$, then for every $\eps >0$ there exists a piecewise affine map $u_\eps \in \Lip(\Omega,\R^n)$ with the following properties:
	\begin{enumerate}[(i)]
		\item\label{or:1} $\|u - u_\eps\|_\infty \le \eps$;
		\item\label{or:2} $u_\eps(x) = Ax + b$ on $\partial \Omega$;
		\item\label{44} $|\{x \in \Omega: \dist(Du_\eps(x),A_i) \le \eps\}| = \lambda_i|\Omega|,\forall i$.
	\end{enumerate}
 
\end{prop}
It is not hard to show that the family $u_\eps$ of Proposition \ref{ind} generates the homogeneous Young measure $\nu = \sum_{i = 1}^{r}\lambda_i\delta_{A_i}$.\ In order to construct a sequence $(\psi_n)_n$ having properties \eqref{item:conv}-\eqref{item:dist}-\eqref{item:e}, it is therefore sufficient to construct a sequence of laminates of finite order $(\nu_n)_n$ having the following properties:
\begin{equation}\label{e:int1}
	\sup_n\int_{\R^{3\times 2}} |\pi_1(X)|^{2p} + |\pi_2(X)|^p d\nu_n < + \infty, \quad \forall p <2,
\end{equation}
where we have used the notation $\pi_i$ introduced in Section \ref{sec:not},
\begin{equation}\label{e:int2}
	\int_{\R^{3\times 2}}\dist(X,K_{e_1}\cup K_{e_2})d\nu_n \to 0,
\end{equation}
and
\begin{equation}\label{e:int3}
	\int_{\R^{3\times 2}}\dist(X,K_{e_i})d\nu_n \ge c > 0,\quad \forall n \in \N,\quad i =1,2.
\end{equation}
Having done so, the existence of the sequence $(\psi_n)_n$ is achieved through Proposition \ref{ind} in conjunction with a diagonal argument. Each $(\nu_n)_n$ is constructed following Faraco's staircase procedure, namely $\nu_{n + 1}$ is obtained from $\nu_n$ by \emph{substituting} a Dirac's delta inside the support of $\nu_n$ with a laminate of finite order $\mu_n$ (a step of the staircase). The construction of $\mu_n$ is the content of Lemma \ref{lem:Zn}, and after that we check that $\nu_n$ enjoys properties \eqref{e:int1}-\eqref{e:int2}-\eqref{e:int3}.

\subsection{Conclusion of the proof of Theorem \ref{t:e}: construction of the staircase laminate}\label{sec:counter1}

We start by constructing $\mu_n$.

\begin{lemma}\label{lem:Zn}
Given an increasing sequence $(a_n)_n \subset \R^+$, set
\[
Z_n = \left(\begin{array}{cc} 0 \\ a_n\id \end{array}\right).
\] 
Then, the discrete measure
\[
\mu_n = \alpha_n\delta_{X_n^1} + \beta_n\delta_{X_n^2} + \gamma_n\delta_{Y_n^1} + \delta_n\delta_{Y_n^2} + \lambda_n\delta_{Z_{n+1}}
\]
with
\[
X_n^1 = \left(\begin{array}{cc} \sqrt{a_n}e_1 \\ a_ne_1\otimes e_1\end{array}\right), \quad X_n^2 = \left(\begin{array}{cc} -\sqrt{a_n}e_1 \\ a_ne_1\otimes e_1\end{array}\right), \quad Y_n^1 = \left(\begin{array}{cc} \sqrt{a_{n+1}}e_2 \\ a_{n+1}e_2\otimes e_2\end{array}\right), \quad Y_n^2 = \left(\begin{array}{cc} -\sqrt{a_{n+1}}e_2 \\ a_{n+1}e_2\otimes e_2\end{array}\right)
\]
and
\[
\alpha_n = \beta_n = \frac{1}{2}\left(1- \frac{a_n}{a_{n + 1}}\right),\quad \gamma_n = \delta_n = \frac{a_n}{a_{n + 1}}\frac{1}{2}\left(1- \frac{a_n}{a_{n + 1}}\right),\quad \lambda_n = \left(\frac{a_n}{a_{n + 1}}\right)^2
\]
is a laminate of finite order with barycenter $Z_n$.
\end{lemma}
\begin{proof}
We start with $\delta_{Z_n}$ and we split it in direction 
\[
\left(\begin{array}{cc} 0 \\ e_2\otimes e_2 \end{array}\right)
\]
to obtain
\[
\mu_n^1 = \alpha_n^1\delta_{X_n} + \alpha_n^2\delta_{E_n}, \text{ where } X_n = \left(\begin{array}{cc} 0 \\ a_ne_1\otimes e_1\end{array}\right) \text{ and } E_n = \left(\begin{array}{cc} 0 \\ a_ne_1\otimes e_1 + a_{n + 1}e_2\otimes e_2\end{array}\right).
\]
 
The resulting weights are
\[
\alpha_n^1 = 1- \frac{a_n}{a_{n + 1}} \text{ and } \alpha_n^2 = \frac{a_n}{a_{n +1}}.
\]
On one hand, we split $\delta_{X_n}$ in the rank-one direction $R$ into the sum
\[
\frac{1}{2}\left(\delta_{X_n^1} + \delta_{X_n^2}\right), \text{ where } R = \left(\begin{array}{cc} e_1 \\ 0 \end{array}\right).
\]
On the other, we split $\delta_{E_n}$ in the rank-one direction $R'$ to obtain a simple laminate of the form
\[
\beta_n^1\delta_{Y_n} +\beta_n^2\delta_{Z_{n + 1}}, \text{ where } R' = \left(\begin{array}{cc} 0 \\ e_1\otimes e_1 \end{array}\right) \text{ and } Y_n =  \left(\begin{array}{cc} 0 \\ a_{n +1}e_2\otimes e_2 \end{array}\right).
\]
A direct computation yields
\[
\beta_n^1 = 1-\frac{a_n}{a_{n + 1}}\text{ and }\beta_n^2 = \frac{a_n}{a_{n + 1}}.
\]
Up to now, we have split $\delta_{Z_{n}}$ in the four-point laminate
\[
\frac{1}{2}\alpha_{n}^1\delta_{X_n^1} + \frac{1}{2}\alpha_{n}^1\delta_{X_n^2} + \alpha_n^2\beta_n^1\delta_{Y_n} + \alpha_n^2\beta_n^2\delta_{Z_{n+1}}.
\]
To conclude, we simply split $\delta_{Y_n}$ into the sum
\[
\frac{1}{2} (\delta_{Y_{n}^1}+\delta_{Y_{n}^2}), \text{ following the rank-one direction } \left(\begin{array}{cc} e_2 \\ 0 \end{array}\right),
\]
and we conclude the construction.
\end{proof}

Having defined the steps $(\mu_n)_n$, we are ready to define the staircase laminates: let $\nu_n$ be defined inductively as $\nu_1 \doteq \mu_1$ and, for $n \ge 1$,
\[
\nu_{n + 1} \doteq \nu_n -p_{n}\delta_{Z_{n+1}} + p_{n}\mu_{n+1}, \quad \text{ where } p_n =\prod_{i = 1}^{n }\lambda_i = \left(\frac{a_1}{a_{n + 1}}\right)^2, p_0 \doteq 1.
\]
 
We choose $a_n \doteq 2^n$. A key feature of the definition of $\nu_{n + 1}$ is that we can compute the value of
\[
\int_{\R^{3\times 2}}f(X)d\nu_n(X)
\]
using a telescoping sum, for any $f \in C^0(\R^{3\times 2})$:
\begin{align*}
	\int_{\R^{3\times 2}}f(X)d\nu_{n + 1} - \int_{\R^{3\times 2}}f(X)d\nu_{1} &= \sum_{i = 1}^n\left(\int_{\R^{3\times 2}}f(X)d\nu_{i + 1} -\int_{\R^{3\times 2}}f(X)d\nu_{i}\right)\\
	&= \sum_{i = 1}^n\left(p_i\int_{\R^{3\times 2}}f(X)d\mu_{i + 1} - p_if(Z_{i + 1})\right)\\
	&= \sum_{i = 1}^np_i\left(\alpha_{i + 1}f(X^1_{i + 1})+\beta_{i + 1}f(X_{i + 1}^2) +\gamma_{i + 1}f(Y^1_{i + 1})+\delta_{i + 1}f(Y_{i + 1}^2) \right)\\
	&\quad\quad +p_{n + 1}f(Z_{n + 2})-p_{1}f(Z_2).
\end{align*}
Hence, using that $\nu_1 = \mu_1$,
\begin{equation}\label{e:exact}
	\begin{split}
	\int_{\R^{3\times 2}}f(X)d\nu_{n + 1} &= \sum_{i = 1}^{n+1}p_{i-1}\left(\alpha_{i}f(X^1_{i})+\beta_{i}f(X_{i}^2) +\gamma_{i}f(Y^1_{i})+\delta_{i}f(Y_{i}^2)\right) + p_{n + 1}f(Z_{n + 2}).
	\end{split}
\end{equation}
With this formula, it is rather simple to rewrite \eqref{e:int1}-\eqref{e:int2}-\eqref{e:int3}. Let us start with \eqref{item:conv}. We observe that on the support of $\nu_n$, the quantities 
\[
|\pi_1(X)|^{2p} + |\pi_2(X)|^p \text{ and } |\pi_2(X)|^p
\]
are comparable with constants independent of $n$.\ Hence it suffices to check that
\begin{equation}\label{e:int4}
	\sup_n\int_{\R^{3\times 2}} |\pi_2(X)|^p d\nu_n \le C(p), \text{ for all }p<2.
\end{equation}
Let $f(X) = |\pi_2(X)|^p$. We further notice that for all $i \in \N$, $f(X_i^1) = f(X_i^2) = a_i^p$ and  $f(Y_i^1) = f(Y_i^2) = a_{i + 1}^p$. Therefore, using \eqref{e:exact} we have
\begin{equation}\label{e:seq1}
	\begin{split}
\int_{\R^{3\times 2}}|\pi_2(X)|^pd\nu_{n + 1} &= \sum_{i = 1}^{n+1}\left(\frac{a_1}{a_{i}}\right)^2\left(\left(1- \frac{a_i}{a_{i + 1}}\right)a_i^p + a_i\left(1- \frac{a_i}{a_{i + 1}}\right)a_{i + 1}^{p-1}\right) + 2^{p/2}\frac{a_1^2}{a_{n + 2}^2}a_{n + 2}^p\\
& \quad \lesssim \sum_{i = 1}^{n+1} 2^{(p-2)i} + 2^{(p-2)(n + 1)}.
\end{split}
\end{equation}
The symbol $a \lesssim b$ means that there is a constant $C > 0$ (independent of $n$) such that $a \le Cb$.\ We now show \eqref{item:dist}.\ Let $f(X) = \dist(X,K_{e_1} \cup K_{e_2})$.\ We have $f(X_n^1) = f(X_n^2) = f(Y_n^1) = f(Y_n^2) = 0$.\ Hence, 
\begin{equation}\label{e:seq2}
\int_{\R^{3\times 2}}\dist(X,K_{e_1} \cup K_{e_2})d\nu_{n + 1} = p_{n + 1}\dist(Z_{n + 2},K_{e_1} \cup K_{e_2}) \le p_{n + 1}|Z_{n + 2}| = \frac{a_1^2}{a_{n+2}}\sqrt{2} \lesssim 2^{-n}.
\end{equation}
Finally, we show \eqref{e:int3}. Let us only consider the case $f(X) = \dist(X,K_{e_1})$, the other one being completely analogous.\ We have $f(X_n^1) = f(X_n^2) = 0$ for all $n$.\ Furthermore, since for any $s \in \R$ and any $n$ we have
\[
\left|Y_n^1-\left(\begin{array}{c}se_1 \\ s^2e_1\otimes e_1\end{array}\right)\right|^2 = \left|Y_n^1\right|^2 +\left|\left(\begin{array}{c}se_1 \\ s^2e_1\otimes e_1\end{array}\right)\right|^2 \ge \left|Y_n^1\right|^2 = a_{n + 1} + a_{n+1}^2 \ge a_{n + 1}^2,
\] 
we find $\dist(Y_n^1,K_{e_1})\ge a_{n + 1}, \forall n \in \N$.\ Hence, using again \eqref{e:exact}:
\begin{equation}\label{e:seq3}
	\begin{split}
		\int_{\R^{3\times 2}}\dist(X,K_{e_1})d\nu_{n + 1} &= \sum_{i = 1}^{n+1}p_{i-1}\left(\gamma_{i}\dist(Y^1_{i},K_{e_1})+\delta_{i}\dist(Y_{i}^2,K_{e_1})\right) + p_{n + 1}\dist(Z_{n + 2},K_{e_1}) \\
		&\ge\frac{1}{8}\sum_{i = 1}^{n+1}p_{i-1}a_{i + 1} = \frac{1}{8}\sum_{i =1}^{n+1}\frac{a_1^2}{a_i^2}a_{i + 1} = \sum_{i = 1}^{n+1} 2^{-i} \ge 2^{-1}. 
	\end{split}
\end{equation}
Now \eqref{e:seq1}-\eqref{e:seq2}-\eqref{e:seq3} show that $(\nu_n)_n$ fulfills \eqref{e:int1}-\eqref{e:int2}-\eqref{e:int3} respectively, and thus prove Theorem \ref{t:e}. \qed
 
\appendix

\section{Chain rule in $L^p$ spaces}\label{ABCp}

In this section we give a proof of Theorem \ref{chain}.\ We will make some additional assumptions to simplify the exposition without loss of generality:

\begin{itemize}
	\item We assume $r < +\infty$, which implies by \eqref{bounds} that $p > 2$. The case $r = + \infty$ and $p = 2$ was recently showed in \cite{GUS};
	\item since the result is local, we can assume that $\Omega = B_1$,\ $u \in L^p(B_1,\R^2)$ and $\beta_i \in L^r(B_1)$ for all $i$;
	\item we can employ \eqref{p1} to infer the existence of $\psi \in W^{1,p}(B_1)$ such that
	\begin{equation}\label{u}
		u = D^\perp \psi \text{ a.e. in }B_1.
	\end{equation}
\end{itemize}

The strategy and the vast majority of the details is the same as \cite{Alberti2014, Bianchini2016, GUS}, and we do not claim any originality in this section. As in \cite{Alberti2014}, we show the following result, from which Theorem \ref{chain} readily follows:

\begin{prop}\label{const}
	Let $\beta \in L^r(B_1)$ and $u \in L^p(B_1,\R^2)$, for $r,p \in [1,+\infty]$.\ Assume \eqref{bounds}-\eqref{u} hold. Then:
	\begin{equation}\label{pde}
		\dv(\beta u) = 0
	\end{equation}
	if and only if, for a.e. $t \in \R$, $\beta$ is constant on every connected component of $\psi^{-1}(\{t\})$.	
\end{prop}

To show this, we need to recall the following characterization of a.e. level set of a continuous Sobolev function. Notice that the continuity of $\psi$ solving \eqref{u} is given by \eqref{bounds} and our assumption $r < +\infty$. The following is taken from \cite{Ntalampekos2020}.

\begin{theorem}\label{char}
	Let $\Omega$ be an open set and let $\varphi \in W_{\loc}^{1,p}(\Omega)$ be a continuous function. Then, for a.e. $t \in \R$, $\varphi^{-1}(\{t\})$ has locally finite $\mathcal{H}^1$ measure and each connected component of $\varphi^{-1}(\{t\})$ is either a point, a Jordan curve or is homeomorphic to an interval.
\end{theorem}

We will need a slightly stronger result.

\begin{corollary}\label{lipchar}
	Let $U \subset \R^2$ be a bounded Lipschitz domain and $\psi \in W^{1,p}(U)$ for $p > 2$. Then:
	\begin{enumerate}
		\item for a.e. $t \in \R$, $\psi^{-1}(\{t\})$ has finite $\mathcal{H}^1$ measure;\label{cor11}
		\item for a.e. $t \in \R$, each connected component $C$ of $\psi^{-1}(\{t\})$ is either a point or admits a Lipschitz arc-length parametrization $\gamma$ such that either
\label{cor12}
		\begin{itemize}
			\item $\gamma: (a,b) \to C \subset U$, is injective on $(a,b)$ and can be extended to $a,b$ with $\gamma(a), \gamma(b) \in \partial U$ or
			\item $\gamma: [a,b] \to C \subset U$, $\gamma(a) = \gamma(b)$ and is injective on $[a,b)$;
			
		\end{itemize}
		\item for those $t$ and for each such curve, letting $D$ be the points of differentiability of $\psi$ and $S$ be the set of points of $D$ at which $D\psi = 0$, $\mathcal{H}^1(\psi^{-1}(\{t\})\cap (S^c \cap D)^c) = 0$ and $\gamma(s) \in D\cap S^c$ for a.e. $s$; \label{cor13}
		\item furthermore, up to a reparametrization of $\gamma$, \label{cor14}
		\begin{equation}\label{tangent1}
			\gamma'(s) = \frac{D^\perp\psi}{|D^\perp\psi|}(\gamma(s)) \quad \text{for a.e. }s \in [a,b].
		\end{equation}
		\item For a.e. $t$, let $\mathcal{C}_t$ be the collection of connected components of $\psi^{-1}(\{t\})$ which are curves, and let $\mathcal{C}_{t}^*$ be the collection of those which are points. Then, setting $B_t \doteq \cup_{C \in \mathcal{C}_t^*}C$, we have $\mathcal{H}^1(B_t) = 0$.\label{it:count}
	\end{enumerate}
\end{corollary}

In turn, Corollary \ref{lipchar} follows from the next Corollary, which applies to functions with compact support defined in the whole plane. To deduce Corollary \ref{lipchar} from the next result, simply extend $\psi$ to a function $\overline{\psi} \in W^{1,p}(\R^2)$ with compact support. The conclusion of Corollary \ref{lipchar} is then straightforward.

\begin{corollary}\label{lipcharcomp}
	Let $\psi \in W^{1,p}(\R^2)$ for $p > 2$ with compact support. Then:
	\begin{enumerate}
		\item for a.e. $t \in \R$, $\psi^{-1}(\{t\})$ has finite $\mathcal{H}^1$ measure;\label{cor1}
		\item for a.e. $t$, each connected component $C$ of $\psi^{-1}(\{t\})$ is either a point or a Jordan curve which admits a Lipschitz parametrization $\gamma: [a,b] \to C$ with $|\gamma'(s)| = 1$ for a.e. $s \in [a,b]$, $\gamma(a) = \gamma(b)$ and $\gamma$ injective on $[a,b)$\label{cor2};
		\item for those $t$ and for each such curve, letting $D$ be the points of differentiability of $\psi$ and $S$ be the set of points of $D$ at which $D\psi = 0$, $\mathcal{H}^1(\psi^{-1}(\{t\})\cap (S^c \cap D)^c) = 0$ and $\gamma(s) \in D\cap S^c$ for a.e. $s$; \label{cor3}
		\item furthermore, up to a reparametrization of $\gamma$, \label{cor4}
		\begin{equation}\label{tangent}
			\gamma'(s) = \frac{D^\perp\psi}{|D^\perp\psi|}(\gamma(s)) \quad \text{for a.e. }s \in [a,b].
		\end{equation}
		\item For a.e. $t$, let $\mathcal{C}_t$ be the collection of connected components of $\psi^{-1}(\{t\})$ which are curves, and let $\mathcal{C}_{t}^*$ be the collection of those which are points. Then, setting $B_t \doteq \cup_{C \in \mathcal{C}_t^*}C$, we have $\mathcal{H}^1(B_t) = 0$.\label{it:count1}
	\end{enumerate}
\end{corollary}
\begin{proof}[Proof of Corollary \ref{lipcharcomp}]
	The proof is divided into steps.\ In all the sets of values $t$ we will define in what follows, we always tacitly exclude $t = 0$.\
	
	\medskip

	\fbox{\emph{Step 1, choice of the good values $t$:}} Let $T_1$ be the set of values for which the conclusion of Theorem \ref{char} holds. For a.e. $t \in T_1$, $\psi^{-1}(\{t\})$ has locally finite $\mathcal{H}^1$ measure from Theorem \ref{char}.\ Since $\psi$ is compactly supported, the measure is actually finite. We let $T_2$ be the set of values $t$ constructed as follows.\ We have $|\R^2\setminus D| = 0$ since $p > 2$. We can then use the coarea formula, see \cite[Theorem 1.1]{MAL}, to write:
	\begin{equation}\label{coarea}
		\int_{\R^2}g|D\psi|dx = \int_{\R}\left(\int_{\psi^{-1}(\{t\})}g d\mathcal{H}^1\right)dt, \text{ for any } g \in L^1(\R^2) \text{ with } |g||D\psi| \in L^1(\R^2).
	\end{equation}
	Taking $g = \chi_{(S^c \cap D)^c}$, we readily find that for a.e. $t \in \R$, 
	\begin{equation}\label{measu}
		\mathcal{H}^1(\psi^{-1}(\{t\})\cap (S^c \cap D)^c) = 0.
	\end{equation}
	Call $T_2$ this set of values $t$. Finally, we need to define a set $T_3$.\ To do so, we start by testing the equation $\dv(D^\perp \psi)= 0$ with a test function given by $\Phi = \varphi(\psi) \eta$, for any $\varphi\in C^\infty_c(\R)$, $\eta \in C_c^1(\R^2)$. As $\psi \in W^{1,p}(\R^2)$, $\varphi(\psi) \eta \in W^{1,p}_0(\R^2)$ is a valid test function. We obtain, by means of the coarea formula,
	\begin{equation}\label{split}
		\begin{split}
			0 &= \int_{\R^2} \left(D^\perp\psi(x), D(\varphi(\psi)\eta)\right)dx  = \int_{\R^2} \varphi(\psi(x))\left(D^\perp\psi(x), D\eta(x)\right)dx\\
			& = \int_{\R^2\cap S^c} \varphi(\psi(x))\left(\frac{D^\perp\psi(x)}{|D^\perp\psi(x)|}, D\eta(x)\right)|D^\perp\psi(x)|dx\\
			& = \int_{\R}\varphi(t)\left(\int_{\psi^{-1}(\{t\})\cap S^c}\left(\frac{D^\perp\psi(x)}{|D^\perp\psi(x)|}, D\eta(x)\right)d\mathcal{H}^1(x)\right)dt\\
			& = \int_{\R}\varphi(t)\left(\int_{\psi^{-1}(\{t\})}\left(\frac{D^\perp\psi(x)}{|D^\perp\psi(x)|}, D\eta(x)\right)d\mathcal{H}^1(x)\right)dt.\\
		\end{split}
	\end{equation}
	In passing to the last inequality we used \eqref{measu}. Hence, there exists a set $G_{\eta}$ of values $t \in T_2$ with $\mathcal{H}^1(\R\setminus G_\eta) = 0$ for which
	\begin{equation}\label{psit}
		\int_{\psi^{-1}(\{t\})}\left(\frac{D^\perp\psi(x)}{|D^\perp\psi(x)|}, D\eta(x)\right)d\mathcal{H}^1(x) = 0.
	\end{equation}
	To get rid of the dependence on $\eta$ of $G_{\eta}$, we can take $R$ sufficiently large such that $\spt(\psi) \subset B_{R/2}$ and consider the separable space $X = C^1_R(\R^2)$ of $C^1(\R^2)$ functions with support in $B_R$, endowed with the topology induced by the norm $\|f\|_{C^0(\R^2)} + \|Df\|_{C^0(\R^2,\R^{2})}$.\ Taking a countable and dense set of $\eta_i \in X$ and the corresponding sets $G_{\eta_i}$ of values $t$ for which \eqref{psit} holds, we let $T_3 \doteq \bigcap_{i\in\N}G_{\eta_i}$. Then, $T_3$ is a set of full measure in $\R$ such that for all $t \in T_3$ and $\eta \in C^\infty_c(\R^2)$, \eqref{psit} holds.\	We can then set $T = T_1\cap T_2 \cap T_3$, which is still of full measure in $\R$.\ This is the set of \emph{good} values of $t$ that we will consider from now on.\ Notice that the intersection with $T_2$ is redundant since $T_3 \subset T_2$, but we added it for clarity.\ We need one last property of the values $t \in T$, namely that, for all $t \in T$ and for all connected component $C \subset \psi^{-1}(\{t\})$,
	\begin{equation}\label{oncc}
		\int_{C} \left(\frac{D^\perp\psi(x)}{|D^\perp\psi(x)|}, D g(x)\right)d\mathcal{H}^1(x) = 0,\quad \forall g \in C^\infty_c(\R^2).
	\end{equation}
 
	To show this we use the fact that any connected component $C$ is the intersection of the closures of decreasing open sets $U_n$ with $\partial U_n \cap \psi^{-1}(\{t\}) = \emptyset$, as noted in \cite[Section 2.8]{ABCa}, see also \cite[Lemma 3.8]{Alberti2014}. We then have $\dist(\partial U_n, \psi^{-1}(\{t\})) = 2\eps_n > 0$, with $\eps_n \to 0$.\ Hence set, for the standard mollifying kernels $\rho_\eps$,
	\[
	\chi_n(x) \doteq (\chi_{U_n}\star \rho_{\frac{\eps_n}{2}})(x). 
	\]
	Notice that $D\chi_n(x) = 0$ over $\psi^{-1}(\{t\})$ since $\dist(\partial U_n, \psi^{-1}(\{t\})) = 2\eps_n > 0$, and that $\chi_n(x) \to \chi_C(x)$ at all points $x$. We then take as a test function in \eqref{psit} $\eta = \chi_n g$, for any $g \in C^\infty_c(\R^2)$, to obtain
	\[
	0 = \int_{\psi^{-1}(\{t\})}\left(\frac{D^\perp\psi(x)}{|D^\perp\psi(x)|}, D(\chi_n g)\right)d\mathcal{H}^1 = \int_{\psi^{-1}(\{t\})}\chi_n(x)\left(\frac{D^\perp\psi(x)}{|D^\perp\psi(x)|}, D g\right)d\mathcal{H}^1, \quad \forall n \in \N.
	\]
	Thus, letting $n \to \infty$, \eqref{oncc} follows.\ We are finally in position to check the validity of the statement of the ongoing corollary.

	\medskip
	
	\fbox{\emph{Step 2, proof of \eqref{cor1}:}} By the definition of $T_1$, $\eqref{cor1}$ is fulfilled.\ 
	
	\medskip
	
	\fbox{\emph{Step 3, proof of \eqref{cor2}:}}
	To see that each connected component which is not a point can be parametrized by an injective Lipschitz parametrization we can invoke\footnote{To apply \cite[Lemma 2.17]{ABCa} we need to exclude triods in the level set $\psi^{-1}(\{t\})$, but this is achieved by the fact that $t \in T_1$, since triods are neither Jordan curves nor homeomorphic to intervals. See \cite[Section 2.3]{ABCa} for the definition of triod.} \cite[Lemma 2.17]{ABCa}. For now, we cannot exclude that the curves may be non-closed. Namely, from \cite[Lemma 2.17]{ABCa}, we find that each component $C$ which is not a point can be parametrized by a Lipschitz, injective parametrization $\gamma: [a,b] \to C$ which is either injective over all $[a,b]$ (namely, it is not a closed curve), or it parametrizes a Jordan curve, i.e. $\gamma(a) = \gamma(b)$ but $\gamma$ is injective over $[a,b)$. We call these curves of type $(A)$ and $(B)$ respectively. When we show \eqref{cor4}, we need to exclude curves of type $(A)$ for $t \in T$. Up to this detail, \eqref{cor2} is shown, provided we choose $\gamma$ parametrized by arclength, which we can do without loss of generality. 
	
	\medskip
	
	\fbox{\emph{Step 4, proof of \eqref{cor3}:}} Concerning \eqref{cor3}, fix any $t \in T$.\ Then, \eqref{measu} shows the first part of the statement.\ To show the second part, we fix any $t \in T$ for which $\psi^{-1}(\{t\})$ has a connected component that is not a point, and we fix one such component $C$, parametrized by $\gamma(s)$. Let 
	\[
	S' = \{s \in [a,b]: \psi \text{ is not differentiable at } \gamma(s) \text{ or it is but $D\psi(\gamma(s)) = 0$}\}.
	\]
	We have
	\begin{equation}\label{measu2}
		\mathcal{H}^{1}(\gamma(S')) \le \mathcal{H}^1(\psi^{-1}(\{t\}) {\cap} (S^c \cap D)^c) \overset{\eqref{measu}}{=} 0.
	\end{equation}
    Since $\gamma$ is injective on $[a,b)$ and $|\gamma'(s)| = 1$ for a.e. $s \in [a,b]$, the area formula \cite[Theorem 3.9]{EVG} gives us that for all measurable $F \subset [a,b]$,
	\[
	\mathcal{L}^1(F) = \int_{F}|\gamma'|(s)ds = \mathcal{H}^1(\gamma(F)).
	\]
	Thus, from \eqref{measu2} we get that $\mathcal{L}^1(S') = 0$. Hence, \eqref{cor3} holds.

	\medskip
	
	\fbox{\emph{Step 5, proof of \eqref{cor4} and absence of type $(A)$ curves:}} From \eqref{cor3} and the equality $\psi(\gamma(s)) = t, \forall s \in [a,b]$,\ we can employ the chain rule to compute:
	\[
	(D\psi(\gamma(s)),\gamma'(s)) = 0, \quad \text{ for a.e. }s \in [a,b].
	\]
	As $D\psi(\gamma(s)) \neq 0$ at a.e. $s$, we obtain the existence of a measurable function $\sigma_C: [a,b] \to \{-1,1\}$ such that
	\begin{equation}\label{gamma'}
		\gamma'(s) = \sigma_C(s)\frac{D^\perp\psi(\gamma(s))}{|D^\perp\psi(\gamma(s))|}.
	\end{equation}
	Starting from \eqref{oncc} and \eqref{gamma'}, we can use the area formula to write:
	\[
	0 \overset{\eqref{oncc}}{=} \int_{C} \left(\frac{D^\perp\psi(x)}{|D^\perp\psi(x)|}, D g\right)d\mathcal{H}^1(x) = \int_{a}^b \left(\frac{D^\perp\psi(\gamma(s))}{|D^\perp\psi(\gamma(s))|}, D g(\gamma(s))\right)ds \overset{\eqref{gamma'}}{=} \int_{a}^b \sigma_C(s)(\gamma'(s), D g(\gamma(s)))ds.
	\]
 
	Due to the fact that $g \in C^\infty_c(\R^2)$ is arbitrary, \cite[Proposition 2.17]{GUS} shows that $\sigma_C(s)$ is constant, and, up to reparametrizing $\gamma$ we can assume $\sigma_C \equiv 1$. Once this is shown, we can still use the previous equality to get
	\[
	g(\gamma(b))-g(\gamma(a)) = 0,
	\]
	which can only be possible if $\gamma(b) =\gamma(a)$. Thus, $\gamma$ was a type $(B)$ curve and this step is showed.
	
	\medskip
	
	\fbox{\emph{Step 6, proof of \eqref{it:count1}:}} We define $G_t \doteq \psi^{-1}(\{t\})\setminus B_t$.\ For $t \in T$, $G_t$ is a countable union of curves (since $\mathcal{H}^1(\psi^{-1}(t)) < + \infty$) hence it is a Borel set. Thus, so is $B_t$.\ 	We fix $\delta > 0$. Start by choosing an open set $A_\delta$ containing $B_t$ with
	\begin{equation}\label{e:adelta}
		\mathcal{H}^1(\psi^{-1}(\{t\})\cap (A_\delta\setminus B_t)) \le \delta.
	\end{equation}
	Every set $\{x_0\}$ with $x_0 \in B_t$ is a connected component of $\psi^{-1}(\{t\})$. Thus, as in Step 1, we find a decreasing sequence of open sets $U_n(x_0)$ such that $\cap_{n}\overline{U_n(x_0)} = \{x_0\}$, and furthermore $\partial U_n(x_0) \cap \psi^{-1}(\{t\}) = \emptyset$ for all $n$. In particular, $\diam(U_n(x_0)) \to 0$. Therefore, if $n(x_0)$ is sufficiently large, we can consider
	\[
	A \doteq \bigcup_{x_0 \in B_t} U_{n(x_0)}(x_0),
	\]
	and achieve
	\[
	B_t \subset A \subset A_\delta,\quad \diam(U_{n(x_0)}) \le \delta  \text{ and } \partial U_{n(x_0)}(x_0) \cap \psi^{-1}(\{t\}) = \emptyset, \quad \forall x_0.
	\]
	Furthermore, we can pick a family of countably many sets $W_m = U_{n(x_m)}(x_m)$ such that $A = \cup_m W_m$. Finally, set $V_m \doteq W_m \setminus (W_{1}\cup\dots \cup W_{m-1})$.\ It can be verified that $V_m$ still enjoys the property that $\partial V_m \cap \psi^{-1}(\{t\}) = \emptyset$. In turn, this allows us to say that
	\begin{equation}\label{e:divm}
	\dv(\tau \mathcal{H}^1\llcorner (\psi^{-1}(\{t\}) \cap V_m)) = 0 \text{ in $\R^2$}, \text{ if } \tau \doteq \frac{D^\perp \psi}{|D^\perp \psi|}.
	\end{equation}
	Indeed, $\partial V_m \cap \psi^{-1}(\{t\}) = \emptyset$ implies that $V_m\cap \psi^{-1}(\{t\})$ and $\psi^{-1}(\{t\})\setminus V_m$ have disjoint closures. Given then any smooth, compactly supported function $\sigma$ which is $1$ on a neighborhood of the closure of the former and $0$ on a neighborhood of the closure of the latter, we can compute, for any $\eta \in C^\infty_c(\R^2)$:
	\[
	\int_{\psi^{-1}(\{t\}) \cap V_m} (D\eta, \tau)d\mathcal{H}^1 = \int_{\psi^{-1}(\{t\}) \cap V_m} (D(\sigma\eta), \tau)d\mathcal{H}^1 = \int_{\psi^{-1}(\{t\})} (D(\sigma\eta), \tau)d\mathcal{H}^1 = 0,
	\]
	the last equality being true since $t \in T_3$.\ Therefore, from \eqref{e:divm} we find the existence of a $BV(\R^2)$ function $u_{m,\delta}$ such that
	\[
	Du_{m,\delta} = \tau^\perp \mathcal{H}^1\llcorner (\psi^{-1}(\{t\}) \cap V_m).
	\]
	Observe that, up to subtracting a constant, we can assume $\spt(u_{m,\delta}) \subset \co(V_m)$ for all $m$, where $\co(V_m)$ is the closure of the convex hull of $V_m$.\ Therefore from H\"older and Sobolev inequalities we find that
	\begin{equation}\label{e:sob}
	\int_{\R^2}|u_{m,\delta}|dx \le \mathcal{L}^2(\spt(u_{m,\delta}))^\frac{1}{2}\left(\int_{\R^2}|u_{m,\delta}|^2dx\right)^\frac{1}{2} \le C\mathcal{L}^2(\co(V_m))^\frac{1}{2}\|Du_{m,\delta}\|(\R^2) \le C\delta\mathcal{H}^1(\psi^{-1}(\{t\})\cap V_m),
	\end{equation}
	where we have used that $\diam(V_m) \le \delta$. We set $u_\delta \doteq \sum_m u_{m,\delta}$, and observe that
	\[
	Du_\delta = \tau^\perp \mathcal{H}^1\llcorner (\psi^{-1}(t)\cap A) \text{ in }\R^2.
	\]
	From \eqref{e:sob}, we find that $\|u_\delta\|_{L^1(\R^2)} \le C\delta$. Thus, for any $\Phi \in C^\infty_c(\R^2,\R^2)$, we get
	\[
	(\tau^\perp \mathcal{H}^1\llcorner B_t)(\Phi) = \int_{B_t}(\tau^\perp,\Phi)d\mathcal{H}^1 = -\int_{(\psi^{-1}(t)\cap A)\setminus B_t}(\tau^\perp,\Phi)d\mathcal{H}^1 - \int_{\R^2}u_\delta\dv(\Phi)dx.
	\]
	Hence, thanks to $\|u_\delta\|_{L^1(\R^2)} \le C\delta$ and \eqref{e:adelta}, we find:
	\begin{align*}
	|(\tau^\perp \mathcal{H}^1\llcorner B_t)(\Phi)| &\le C(\Phi)\left(\mathcal{H}^1((\psi^{-1}(t)\cap A)\setminus B_t) + \|u_\delta\|_{L^1}\right)\\
	&\le C(\Phi)\left(\mathcal{H}^1((\psi^{-1}(t)\cap A_\delta)\setminus B_t) + \|u_\delta\|_{L^1}\right) \le C(\Phi)\delta
	\end{align*}
	We then find that $(\tau^\perp \mathcal{H}^1\llcorner B_t)(\Phi) = 0$ for all $\Phi \in C^\infty_c(\R^2)$, and hence that $\tau^\perp \mathcal{H}^1\llcorner B_t =0$. This implies that $\mathcal{H}^1(B_t) = 0$, as wanted.
\end{proof}

We are finally in position to show Proposition \ref{const}.

\begin{proof}[Proof of Proposition \ref{const}]
	Assume first the PDE \eqref{pde} holds. This means that
	\begin{equation}\label{pdeweak}
		0 = \int_{B_1} \beta(x)(D^\perp\psi(x), D\Phi(x))dx,\quad \text{for all }\Phi \in C^\infty_c(B_1).
	\end{equation}
	We claim that there exists a set of values $t \in T' \subset \R$, with $\mathcal{L}^1(\R \setminus T') = 0$ for which, on any connected component $C \subset \psi^{-1}(\{t\})$,
	\begin{equation}\label{Cbeta}
		0 = \int_{C} \beta(x)\left(\frac{D^\perp\psi}{|D^\perp\psi|}(x), D\eta(x)\right)d\mathcal{H}^1,\quad \text{for all }\eta \in C^\infty_c(B_1).
	\end{equation}
	The proof of this property is precisely where we use our assumption \eqref{bounds}, but, since the computations are entirely analogous to those of the last part of Step 1 of the proof of Corollary \ref{lipcharcomp}, we will only sketch it.\ We consider the test function $\Phi_\eps = \varphi(\psi_\eps) \eta$, for any $\varphi \in C^\infty_c(\R)$, $\eta \in C^\infty_c(B_1)$, where $(\psi_\eps)_\eps$ is a family of smooth functions converging in $W^{1,p}(B_1)$ to $\psi$.\ We can plug in $\Phi_\eps$ in the weak form of the equation $\dv(\beta D^\perp \psi) = 0$ to obtain 
	\begin{equation}\label{split2}
		\begin{split}
			0 &= \int_{\R^2}\beta \left(D^\perp\psi, D(\varphi(\psi_\eps)\eta)\right)dx  = \int_{\R^2} \beta\varphi(\psi_\eps)\left(D^\perp\psi, D\eta\right)dx + \int_{\R^2}\beta \varphi'(\psi_\eps)\eta\left(D^\perp\psi, D\psi_\eps\right)dx.
		\end{split}
	\end{equation}
	Since $p > 2$, we find immediately that the first addendum converges to
	\begin{equation}\label{eq:final}
	\int_{\R^2} \beta(x) \varphi(\psi(x))\left(D^\perp\psi(x), D\eta(x)\right)dx.
	\end{equation}
	We need to show that the second one converges to zero.\ To see this, observe that $\eqref{bounds}$ implies that $\beta |D\psi|^2 \in L^1$, and hence, for some constant $C > 0$,
\begin{align*}
	 \left|\int_{\R^2} \beta \varphi'(\psi_\eps )\eta \left(D^\perp\psi \right.\right., &\left.\left. D\psi_\eps \right)dx\right| = \left|\int_{\R^2} \beta \varphi'(\psi_\eps )\eta \left(D^\perp\psi , D\psi_\eps -D\psi \right)dx\right| \\
	 & \le C\int_{B_1}|\beta||D\psi||D\psi_\eps-D\psi|dx \\
	 &\le C \left(\int_{B_1}|\beta||D\psi|^2dx\right)^\frac{1}{2}\left(\int_{B_1}|\beta||D\psi - D\psi_\eps|^2dx\right)^\frac{1}{2}\\
	 & \le C\left(\int_{B_1}|\beta||D\psi|^2dx\right)^\frac{1}{2}\left(\int_{B_1}|\beta|^r\right)^\frac{1}{2r}\left(\int_{B_1}|D\psi - D\psi_\eps|^{2r'}dx\right)^\frac{1}{2r'},
\end{align*}
where $r' = \frac{r}{r-1}$ is the conjugate exponent of $r$.\ Due to our assumption \eqref{bounds} $2r' \le p$, and hence the last integral in \eqref{split2} converges to $0$. Combining this with \eqref{eq:final}, \eqref{split2} shows
\[
\int_{\R^2} \beta(x) \varphi(\psi(x))\left(D^\perp\psi(x), D\eta(x)\right)dx = 0, \quad \text{ for any $\varphi \in C^\infty_c(\R)$, $\eta \in C^\infty_c(B_1)$.} 
\]
Now we can apply verbatim the computations of Step 1 of Corollary \ref{lipcharcomp} to obtain our claim \eqref{Cbeta}.\ Once this is achieved, we can exclude the trivial case where $C$ is a single point, and hence we can use the parametrization of $C$ provided by \eqref{cor12}-\eqref{cor13}-\eqref{cor14} of Corollary \ref{lipchar} and the area formula to rewrite
	\begin{equation}\label{areaC}
		\int_{C} \beta(x)\left(\frac{D^\perp\psi}{|D^\perp\psi|}(x), D\eta\right)d\mathcal{H}^1 = \int_a^b \beta(\gamma(s))(\gamma'(s),D\eta(\gamma(s)))ds.
	\end{equation}
	\eqref{Cbeta} implies that the left-hand side of \eqref{areaC} is $0$, and \cite[Proposition 2.17]{GUS} yields that $\beta(\gamma(s))$ is constant, as wanted.
	
	\medskip
	
	Assume now that $\beta$ is constant on every connected component of a.e. level set of $\psi$.\ Fix any $\eta\in C^\infty_c(B_1)$.\ Then, if $C$ is not a single point and $\beta \equiv \beta_C$ on $C$, \eqref{areaC} still implies that
	\[
	\int_{C} \beta(x)\left(\frac{D^\perp\psi}{|D^\perp\psi|}(x), D\eta\right)d\mathcal{H}^1 = \int_a^b \beta(\gamma(s))(\gamma'(s),D\eta(\gamma(s)))ds = \beta_C(\eta(\gamma(b))-\eta(\gamma(a))) = 0,
	\]
	the last equality being true in view of Corollary \ref{lipchar}\eqref{cor12}. In turn, using the notation and conclusion of Corollary \ref{lipchar}\eqref{it:count}, we can consider the (at most) countably many curves $C_i \in \mathcal{C}_t$ and, recalling that $\mathcal{H}^1(B_t) = 0$, write
 
	\begin{equation}\label{eq:lastzer}
	\int_{\psi^{-1}(\{t\})} \beta(x)\left(\frac{D^\perp\psi}{|D^\perp\psi|}(x), D\eta\right)d\mathcal{H}^1 = \sum_{i}\int_{C_i} \beta(x)\left(\frac{D^\perp\psi}{|D^\perp\psi|}(x), D\eta\right)d\mathcal{H}^1 = 0.
	\end{equation}
	Therefore, recalling $S$ from Corollary \ref{lipchar}\eqref{cor13} and using once again the coarea formula \eqref{coarea}:
	\begin{align*}
		\int_{B_1} \beta(x)(D^\perp\psi(x), D\eta)dx &= \int_{B_1\cap S^c} \beta(x)\left(\frac{D^\perp\psi}{|D^\perp\psi|}(x), D\eta\right) |D^\perp\psi|dx \\
		&=\int_\R \left(\int_{\psi^{-1}(\{t\})\cap S^c} \beta(x)\left(\frac{D^\perp\psi}{|D^\perp\psi|}(x), D\eta\right)d\mathcal{H}^1\right) dt \\
		&=\int_\R \left(\int_{\psi^{-1}(\{t\})} \beta(x)\left(\frac{D^\perp\psi}{|D^\perp\psi|}(x), D\eta\right)d\mathcal{H}^1\right) dt \overset{\eqref{eq:lastzer}}{=} 0,
	\end{align*}
	where to pass from the second to the third line we used Corollary \ref{lipchar}\eqref{cor13}.\ The proof is then concluded.
\end{proof}

\bibliographystyle{plain}
\bibliography{BiblioPressureless}

\end{document}